\documentclass[12pt, a4paper]{article}
\usepackage{jansstylefile} 
\title{
	Adaptive posterior contraction rates for  empirical Bayesian drift estimation of a diffusion %
}
\author{Jan van Waaij\thanks{University of Padova, Department of Statistics, \href{mailto:jvanwaaij@gmail.com}{jvanwaaij@gmail.com}. The author is supported by the University of Padova under the STARS Grant. This paper is an expanded and improved version of chapter five of the author's PhD-thesis (\cite{waaij2018}). The author is happy to receive feedback.}}

\begin{document}	
\begin{abstract}
Due to their conjugate posteriors,
Gaussian process priors are attractive
 for estimating the drift of 
 stochastic differential equations
 with continuous time observations. 
 However, their performance strongly 
 depends on the choice of the hyper-parameters. 
 We employ the marginal maximum likelihood estimator to estimate the scaling and/or smoothness parameter(s) of the prior and show that the corresponding posterior has optimal rates of convergence. General theorems do not apply directly to this model as the usual test functions are with respect to a random Hellinger-type metric.

We allow for continuous and discrete, one- and two-dimensional sets of hyper-parameters, where optimising over the two-dimensional set of smoothness and scaling hyper-parameters is shown to be beneficial in terms of the adaptive range. 
\end{abstract}

\subsubsection*{Summary}

We assume a continuous time observation \(X^T=(X_t:t\in[0,T]), T>0\), which is a realisation of the stochastic differential equation (SDE) \(X_t=\theta_0(X_t)dt+dW_t,\) with \(W_t\) a Brownian motion. We are interested in estimating \(\theta_0\), which is only assumed to be 1-periodic, zero-mean (\(\int_0^1\theta_0(x)dx=0\)) and \(\beta\)-Sobolev smooth for some unknown \(\beta>0\). We equip the space of 1-periodic \(L^2\)-functions (on \([0,1]\)) with a Gaussian process prior, given by \(s\sum_{k=1}^{\floor T}k^{-\alpha-1/2}Z_k\phi_k\), where the hyper-parameters \(s\) and/or \(\alpha\) are estimated from the data using the marginal maximum likelihood estimator. Inference is done with a posterior. Fixing \(s\) and optimising over \(\alpha\in A=[1/2+\delta,\sqrt{\log T}]\) (for a small fixed constant \(\delta>0\)) leads to the optimal rate of posterior convergence \(T^{-\frac\beta{1+2\beta}}\), as long as the smoothness \(\beta\) of the true parameter is at least \(1/2+\delta\). Fixing \(\alpha>1/2\) and optimising over \(s\in S(\alpha)=\vh{T^{-\frac1{4+4\alpha}},T^\alpha}\) leads to optimal rates as long as \(0<\beta\le \alpha+1/2\). Optimising over \(\alpha\in A\) and \(s\in S(\alpha)\) simultaneously leads to optimal rates for all \(\beta>0\). 
The same results are obtained when \(A\) and \(S(\alpha)\) are replaced with a finite discrete grid that is `dense' enough in \(A\) and \(S(\alpha)\), respectively. 
For the proof we use the approach of \cite{donnetrivoirardrousseauscricciolo2014} and        \cite{rousseauszabo2017}, adapted to the SDE model. 

	\section{Introduction}
	
 Nonparametric Bayesian procedures for stochastic differential equations (SDEs) have recently received significant attention, which is partly motivated by the many applications SDEs have in science and economics (e.g. \cite{ditlevsenditlevsenandersen2002}, \cite{gottwald_crommelin_franzke_2017}  climate research,  \cite{Hindriks} neurobiology, and \cite{karatzasshreve1998}, \cite{gourierouxea2017} finance), and the many fast numerical schemes that have been developed for Bayesian nonparametric procedures for SDEs, see \cite{Papa}, \cite{Moritz}, \cite{meulenschauerzanten2017}, \cite{vandermeulenschauer2017} and \cite{meulenschauer2018}, for example, and the Yuima R-package \cite{yuima} for statistical software and the Julia-packages \cite{SchauderJulia} and  \cite{waaijJulia}. 

	This motivates the study of the 
	theoretical performance of these 
	nonparametric Bayesian 
	procedures. Posterior consistency 
	or posterior contraction rates for 
	nonparametric Bayesian models for 
	diffusions are examined in 
	\cite{meulen2006}, 
	\cite{PanzarZanten2009}, 
	\cite{meulenzanten2013}, 
	\cite{gugushvilispreij2014}, 
	\cite{waaijzanten2017}, 
	\cite{nicklray2018}, 
	\cite{meulenschauerwaaij2018}, and 
	\cite{koskelaspanojenkins2019}. To 
	the best of the author's knowledge, no 
	attention has been given to 
	empirical Bayes methods for 
	diffusions and their asymptotic 
	properties, which is the subject of 
	the current study. 
	
	We assume a given continuous time observation  \(X^T=\rh{X_t:t\in[0,T]}, T>0\) that is a weak solution to the stochastic differential equation 
	\begin{align}
		dX_t=\theta_0(X_t)dt+dW_t,\label{eq:sde}
	\end{align}
	where \(W_t\) is a Brownian motion (BM) and the unknown parameter \(\theta_0\) is a one-periodic real-valued function on \(\re\) with \(\int_0^1\theta(x)dx=0\)  and \(\int_0^1\theta_0(x)^2dx<
	\infty\). Equivalently, \(\theta_0\) can be seen as a measurable square integrable function on the unit circle, with zero mean. The space of this functions is denoted by \(\mathring L^2(\TT)\). 
	
	In molecular dynamics, \cref{eq:sde} (\cite{Papa}) is used to model the angle between atoms in a molecule. \cite{Papa} proposes a Gaussian process prior on the unknown drift and a numerical scheme to sample from the posterior is given. In \cite{PSZ}, posterior consistency is shown for this prior. In \cite{waaijzanten2016}, minimax posterior convergence rates are shown for the prior whose law is defined by the random function
	\begin{equation}\label{eq:seriesexpansion}
		\theta=s\sum_{k=1}^\infty k^{-1/2-\alpha}Z_k\phi_k, \quad Z_k\stackrel{\text{i.i.d.}}{\sim}N(0,1),
	\end{equation}
	for fixed positive \(s\) and \(\alpha\) and when \(\theta_0\) is \(\alpha\)-Sobolev smooth. In addition, if \(s\) is equipped with a specific (hyper)prior, then adaptation with minimax rates (up to a multiplicative constant) are shown as long as the Sobolev-smoothness of the true parameter is at most \(\alpha+1/2\), with fixed \(\alpha\). Posterior convergence rates are shown for every Sobolev smoothness when \(\alpha\) is equipped with a certain hyperprior, and \(s\) is set to one. General conditions for posterior contraction for continuously observed diffusions with several examples are examined in \cite{meulen2006}. Nonparametric Bayesian procedures with discrete observations of \cref{eq:sde} are considered in \cite{nicklsohl2017} (low frequency regime) and \cite{abraham2019} (high frequency regime). Nonparametric drift estimation is studied from the frequentist perspective in \cite{hoffmann1999}, \cite{spokoiny2000}, \cite{zanten2001},  \cite{dalalyankutoyants2002}, \cite{gobethoffmannreiss2004}, \cite{dalalyan2005}, and  \cite{comtegenoncatalotrozenholc2007}, among others.  

	Going back to the Bayesian procedure, one quite often sees that the hyperprior destroys the appealing numerical properties of the prior as the hierarchical prior is no longer conjugate, for example. Therefore, in practise, one often relies on  empirical Bayes methods, in which the hyper-parameters of the prior are estimated from the data and then plugged into the prior. Inference is done via the corresponding posterior.
	
General methods to study posterior contraction for empirical Bayes procedures in different models with the marginal maximum likelihood estimator (MMLE or MML) as an estimator for the hyper-parameter(s) are examined in \cite{donnetrivoirardrousseauscricciolo2014}, \cite{rousseauszabo2017}, and \cite{martinwalker2019}. We follow the approach of \cite{donnetrivoirardrousseauscricciolo2014} and        \cite{rousseauszabo2017}, the difference in the current study is that the tests to distinguish between the  corresponding drift functions are in the random Hellinger metric (as introduced in \cite{meulen2006}), which is with high probability equivalent to the \(L^2\)-metric. 
 The basic idea is first to determine a `good set' of hyper-parameters that contain the  MMLE with high probability, 
 and that the corresponding posteriors associated with this good set converge uniformly with the desired 
 rate, under the true parameter. We consider  MML estimators of the smoothness and scaling parameters, as well as an MML estimator for two-dimensional hyper-parameters \(\lambda=(\alpha,s )\) simultaneously. Hyper-parameters in different 
 parts of the hyper-parameter space may lead to optimal rates; for example, for rough truths, small \(\alpha\) and 
 moderate \(s\), or larger \(\alpha\) and larger \(s\) values. This may lead to large or `strange' formed sets 
 \(\Lambda_0\) of good parameters. The general theorem of \cite{rousseauszabo2017} 
 allows for two and higher dimensional hyper-parameters, but they do not provide an example of such a prior. To the best of our knowledge, we are the first to do so using their approach. 
However, we find that this does not lead to complications for the analysis, but it does improve the results as it 
leads to an optimal rate for every \(\beta\)-Sobolev smooth function for all \(\beta>0\), where the (proven) adaptive 
range is smaller when we only optimise over \(\alpha\) or \(s\) with the other hyper-parameter fixed.

	\section{The marginal maximum likelihood estimator and the empirical posterior}
	
	For a family of priors \(\set{\Pi_\lambda:\lambda\in\Lambda}\), we define the marginal likelihood as \begin{align*}
		\int p^\theta(X^T)d\Pi_\lambda(\theta),
	\end{align*}
	where \begin{align*}
		 p^\theta (X^T)=\expa{\int_0^T \theta(X_t)dX_t-\frac12 \int_0^T\theta(X_t)^2dt}
	\end{align*}
	is the density 	relative to the Wiener measure. 
	
	The maximum marginal likelihood estimator (MMLE) is a \(\hat \lambda\) in \(\Lambda\) that maximises the marginal  likelihood, 
	\begin{equation}\label{eq:MMLE}
		\hat\lambda=\arg\max_{\lambda\in \Lambda}\int p^\theta(X^T)d\Pi_\lambda(\theta).
	\end{equation}
	As such maximum may not exist, or is difficult to find exactly, we also allow any choice of \(\hat\lambda\) in \(\Lambda\) for which \begin{equation}\label{eq:alternativechoiceforhatlambda}
\int p^\theta(X^T)d\Pi_{\hat \lambda}(\theta)\ge \frac12\sup_{\lambda\in \Lambda}\int p^\theta(X^T)d\Pi_\lambda(\theta).	
	\end{equation}
	Obviously, a \(\hat\lambda\) satisfying \cref{eq:MMLE} also satisfies \cref{eq:alternativechoiceforhatlambda}, so we work with \(\hat\lambda\) satisfying \cref{eq:alternativechoiceforhatlambda} from now on. The choice for the factor \(1/2\) is, of course, somewhat arbitrary, but is chosen to keep the analysis and notation simple.
	
	For fixed \(\lambda\), the posterior of a measurable set \(A\subseteq \mathring L^2(\TT)\) is given by
	\begin{align}\label{eq:posterior}
		\Pi_\lambda(A\mid X^T)=\frac{\int_A  p^\theta(X^T)d\Pi_\lambda(\theta)}{\int  p^\theta(X^T)d\Pi_\lambda(\theta)}.
	\end{align}
	The posterior is well defined, see \cite[lemma 3.1]{waaijzanten2016}.
We study posterior convergence rates of \(\Pi_{\hat\lambda}(\sdot\mid X^T)=\Pi_{\lambda}(\sdot\mid X^T)\Big|_{\lambda=\hat\lambda}\).  We refer to \(\Pi_{\hat\lambda}(\sdot\mid X^T)\) as the empirical posterior, and to \(\Pi_{\hat\lambda}(\sdot)\) as the empirical prior. 

The family of Gaussian priors we consider is displayed below, where \(\{\phi_k\}_{k=1}^\infty\) is the orthogonal Fourier basis of \( \mathring  L^2(\TT)\) defined by \(		\phi_{2k-1}(x)=\sqrt2\sin(2\pi kx)\), \( 
		\phi_{2k}(x)=\sqrt2\cos(2\pi k x),\)  \(k\in \NN 
	\) 
 and \(\set{Z_i:i\in \NN}\) are independent standard normally distributed random variables. The prior \(\Pi_\lambda\) is defined as the law of the random process defined by \begin{equation}
 t\mapsto s\sum_{ k=1}^{\floor T} k^{-\alpha -1/2}Z_k\phi_k(t),
 \end{equation}
 where \( \lambda= (\alpha, s)\) is an element of \(\Lambda \), which is a subset of  \begin{equation}\label{eq:defLambda}
\Lambda_1:=\ah{(\alpha,s):1/2+\delta\le \alpha\le \sqrt{\log T},\quad T^{-\frac1{4+4\alpha}}\le s\le  T^\alpha},  
 \end{equation}
 and where \(\delta>0 \) is a small fixed constant. The condition that \(\alpha\) is bounded away from \(1/2\) is found to be necessary in our approach to derive a test that distinguishes between drift functions. It would be worthwhile if future research could alleviate this condition. 
For ease of  notation, we suppress the dependence on \(T\).

We say that the posterior contracts with rate \(\eps_T\downarrow 0\), as \(T\to\infty\), when for some constant \(G>0\) \begin{align*}
		\Pi_{\hat\lambda}\rh{\theta:\|\theta-\theta_0\|\le G\eps_T\mid X^T}\to 1,
	\end{align*}
	in \(\P^{\theta_0}\)-probability as \(T\to\infty\).

	We obtain the following result: 
	\begin{theorem}\label{thm:posteriorcontractiontheorem}
Let \(\theta_0\) be \(\beta\)-Sobolev smooth, that is, \(\theta_0\in \mathring L^2(\TT)\) is of the form \(
\theta_0=\sum_{ k=1}^\infty \theta_k\phi_k \) and satisfies \(\sum_{ k=1}^\infty k^{2\beta}\theta_k^2<\infty. 
\) Then the posterior contracts with rate \( T^{-\frac\beta{1+2\beta}}, \) when
\begin{enumerate}
	\item \(\Lambda = \set{(\alpha,s):T^{-\frac1{4+4\alpha}}  \le s\le T^\alpha }\), for some fixed \(\alpha>1/2\), and \( \beta\le \alpha+1/2,\) \emph{or},
	\item \(\Lambda = \set{(\alpha,s): 1/2+\delta \le \alpha \le \sqrt{\log T}}\) for some fixed \(s>0\), and when \(\beta\ge 1/2+\delta\), \emph{or},
	\item \(\Lambda=\Lambda_1\) and \(\beta>0. \)
\end{enumerate}
The same results hold when \(\Lambda\) in each of the three cases is replaced by the discrete sets \begin{enumerate}[1']
\item \(\Lambda=\set{(\alpha, lT^{-\frac1{4+4\alpha}}):l\in \NN, 1\le l\le T^{\alpha+\frac1{4+4\alpha}}},\) for some fixed \(\alpha>1/2\)
\item \(\Lambda=\set{(k/\log T, s):k\in\NN,(1/2+\delta)\log T\le k\le (\log T)^{3/2}}\), for some fixed \(s>0\),
\item \(\Lambda=\Bigl\{(k/\log T, lT^{-\frac1{4+4k/\log T}}):k,l\in\NN,(1/2+\delta)\log T\le k\le (\log T)^{3/2},\\ \phantom{\Lambda=\Bigl\{(k/\log T, lT^{-\frac1{4+4k/\log T}}):k,l\in\NN,(1/2+\delta) 
}1\le l\le T^{k/\log T+\frac1{4+4k/\log T}}\Bigr\}\),
\end{enumerate} 
respectively.   
\end{theorem}Note that similar results are obtained in \cite{waaijzanten2016} 
for a hyperprior on the scaling (case 1) and a hyper-parameter on the 
smoothness parameter (case 2), with the difference that in the latter case adaptivity for all 
\(\beta>0 \) is obtained (\cite[theorem 3.4]{waaijzanten2016}), which, in our study, is only 
obtained when \(\alpha \) and \(s\)  are estimated at the same time.
		
	\section{The proof of \cref{thm:posteriorcontractiontheorem}}\label{sec:outlineoftheproof}

	Using the equivalence of the measures \(P^\theta\), for fixed \(T\) (see \cite[lemma 13]{meulenschauerwaaij2018}), the posterior \cref{eq:posterior} can be written as 
	\begin{align*}
	\Pi_\lambda(A\mid X^T)=&\frac{\int_A p^\theta(X^T)/p^{\theta_0}(X^T)d\Pi_\lambda(\theta)}{\int p^\theta(X^T)/p^{\theta_0}(X^T)d\Pi_\lambda(\theta)}	,
	\intertext{and}
	\bar p^\theta(X^T):=&p^\theta(X^T)/p^{\theta_0}(X^T)\\
	=&\expa{\int_0^T(\theta(X_t)-\theta_0(X_t))dW_t-\frac12\int_0^T(\theta(X_t)-\theta_0(X_t))^2dt}
	\end{align*}
	is the density of \(P^\theta\) relative to \(P^{\theta_0}\).

	Since \(\bar p^\theta(X^T)=p^\theta(X^T)/p^{\theta_0}(X^T)\) differs from \(p^\theta(X^T)\) only by a positive multiplicative constant not depending on \(\theta\) or \(\lambda\), the \(\lambda\) that satisfies \cref{eq:alternativechoiceforhatlambda} satisfies also  \cref{eq:alternativechoiceforhatlambda} with \(p^\theta\) replaced by \(\bar p^\theta\).

	For technical reasons, we introduce a positive constant \(K\) (which does not depend on \(\lambda\) or \(T\), but may depend on \(\theta_0\)). It is easily shown that for some  unique \(\eps_\lambda\) (dependent on \(K, T\) and \(\theta_0 \))
	
	\begin{align}
		\Pi_\lambda(\|\theta-\theta_0\|_2<K\eps_\lambda)=e^{-T\eps_\lambda^2}.\label{eq:definingpropertyepslambda}
	\end{align}
	Let \(\eps_0=\inf_{\lambda\in\Lambda} \eps_\lambda\). Intuitively, we can think of the prior \(\Pi_\lambda\) that minimises \(\eps_\lambda\) as one that most favours the true parameter out of all priors, by setting the maximum prior mass around \(\theta_0\) 
	and is therefore expected to give the best possible rates of posterior convergence.
	
	For a constant \(M\ge 2\), we define 
	\begin{align}\label{eq:definitionLambdazero}
		\Lambda_0=\set{\lambda\in\Lambda:\eps_\lambda\le M\eps_0}
	\end{align}
	as the set of all  \(\lambda\in\Lambda\) where the prior \(\Pi_\lambda\) puts a significant amount of prior mass around \(\theta_0\). Note that this set is non-random but does depend on the unknown true parameter \(\theta_0\).
	In the next section, \cref{sec:hatlambdainlambdazerowithprobabilityconvergingtoone}, we prove the existence of a sequence of events \(F\) (depending on \(T\)), which asymptotically have \(\P^{\theta_0}\)-probability one, and on which \(\hat\lambda\in \Lambda_0\). Obviously, for a constant \(H>1\), we have
	\begin{align*}
		&\E_{\theta_0}\left[\Pi_{\hat\lambda}(\|\theta-\theta_0\|_2\ge HM\eps_0  \mid X^T)\right]
		\le \E_{\theta_0}\left[\sup_{\lambda\in\Lambda_0}\Pi_{\lambda}(\|\theta-\theta_0\|_2\ge HM\eps_0  \mid X^T)\II_ F\right]+\P^{\theta_0}(F^c). 
	\end{align*}
	
	In \cref{sec:posteriorconvergencegivenhatlambdainLambdazero} we show that for sufficiently large \(H\), \[\E_{\theta_0}\left[\sup_{\lambda\in\Lambda_0}\Pi_{\lambda}(\|\theta-\theta_0\|_2\ge HM\eps_0  \mid X^T)\I_F\right]\to0\]
	as \(T\to\infty\). For each respective choice of \(\Lambda\) in \cref{thm:posteriorcontractiontheorem}, the upper bounds for \(\eps_0\) in \cref{sec:upperandlowerboundsforepslampda} then allows us to conclude the theorem. 
	
	\section{The asymptotic behaviour of the marginal maximum likelihood estimator}\label{sec:hatlambdainlambdazerowithprobabilityconvergingtoone} 
	
	In \cite[theorem 4.1]{PSZ}, it is shown that for a high probability event the Hellinger distance 
	\begin{align*}
		h(\theta,\theta')= \sqrt{\int_0^T(\theta(X_u)-\theta'(X_u))^2du}
	\end{align*}
	and the square integrable (\(L^2\)-) norm is equivalent. To be precise, for any constants \(0<c_\rho<1<C_\rho\) satisfying 
	
	\begin{align*}
		c_\rho^2 <& \inf_{x\in[0,1]}\rho(x)\le \sup_{x\in[0,1]}\rho(x)<C_\rho^2
		\intertext{with}
		\rho(x)=&\frac{\expa{2\int_0^x\theta_0(y)dy}}{\int_0^1\expa{2\int_0^y\theta_0(z)dz}dy}
	\end{align*}
	the \(P^{\theta_0}\)-probability of the event 
	\begin{equation}\label{eq:definitionofE}
		E=\set{c_\rho\sqrt T\|\theta-\theta'\|_2\le h(\theta,\theta')\le C_\rho\sqrt T\|\theta-\theta'\|_2,\forall \theta,\theta'\in \mathring L^2(\TT)}
	\end{equation}
converges to 1, as \(T\to\infty\).

	 Let \(A<B\) be positive constants. Define the event (which depends on \(T\))  
	 \begin{align*}
	 	F:=& E\cap\set{\sup_{\lambda\in \Lambda_0^c}\int \bar p^\theta(X^T)d\Pi_\lambda(\theta)\le e^{-BT\eps_0^2}}
	 	\cap \set{\sup_{\lambda\in\Lambda_0}\int \bar p^\theta(X^T)d\Pi_{\lambda}(\theta)\ge e^{-AT\eps_0^2}},
	 \end{align*}
	 where \(\Lambda_0^c=\Lambda\weg\Lambda_0\) is the complement of \(\Lambda_0\) in \(\Lambda\). 
	 
	 As \(T\eps_0^2\to\infty\) (see \cref{lem:Tepszerosquaredconvergestoinfinity}) and \(A<B\),  it follows that for large \(T\), in the event \(F\),
	 \begin{align*}
	 \int \bar p^{\theta}(X^T)d\Pi_{\hat \lambda}(\theta)\ge \frac12 \sup_{\lambda\in\Lambda}\int \bar p^\theta(X^T)d\Pi_{\lambda'}(\theta)
	 \ge \frac12\sup_{\lambda\in\Lambda_0}\int \bar p^\theta(X^T)d\Pi_{\lambda}(\theta)\ge \frac12e^{-AT\eps_0^2}> e^{-BT\eps_0^2}. 
	 \end{align*}
	Which implies that for large \(T\), in the event \(F\),  the estimator \(\hat\lambda\) is in the set of good hyper-parameters \(\Lambda_0\).

	 \begin{theorem}\label{thm:hatlambdainafavourableset}
	 	There are constants \(0<A<B\), so that when \(M\) in \cref{eq:definitionLambdazero} is sufficiently large, \(P^{\theta_0}(F)\to1\) as \(T\to\infty\). Moreover, for every \(\lambda'\in\Lambda_0\) satisfying \(\eps_\lambda\le2\eps_0\), we have in the event \(F\) \[\int \bar p^\theta(X^T)d\Pi_{\lambda'}(\theta)\ge e^{-AT\eps_0^2}.\]
	 \end{theorem}
	 
	 \begin{proof}
	 	 
	 	 By the construction of \(\Lambda_0\), there is a \(\lambda'\in\Lambda_0\) with 
	 	\(\eps_{\lambda'}\le 2\eps_0\) (and clearly \(\eps_{\lambda'}\ge \eps_0\)). Using \cite[lemma 4.2]{meulen2006} and  \(\eps_0\le \eps_{\lambda'}\le 2\eps_0\), we see that for every such \(\lambda'\),
	 	
	 	\begin{align*}
	 		P^{\theta_0}\rh{\ah{\int \bar p^\theta(X^T)d\Pi_{\lambda'}(\theta)>e^{-4(C_\rho^2K^2+1)T\eps_0^2}}\cap E}
	 		\ge  P^{\theta_0}(E)-\expa{-\frac18C_\rho^2T\eps_0^2}.
	 	\end{align*}
	 	
	 	As \(\P^{\theta_0}(E)\) converges to one, and the other term to zero, by  \cref{lem:Tepszerosquaredconvergestoinfinity}, it follows that the quantity on the left of the inequality converges to one.

	 	Let \(B>4(C_\rho^2K^2+1)\) be a constant. Let us now concentrate our attention on the event 
	 	
	 	\begin{equation}\label{eq:eventthattheMLissmalloutsidelambdazero}
	 		E\cap\set{\sup_{\lambda\in \Lambda_0^c}\int \bar p^\theta(X^T)d\Pi_\lambda(\theta) > e^{-BT\eps_0^2}}.
	 	\end{equation}
	 	
	 	We have to show that the \(P^{\theta_0}\)-probability of    \cref{eq:eventthattheMLissmalloutsidelambdazero} converges to zero. We have\begin{align*}
	 		\sup_{\lambda\in \Lambda_0^c}\int \bar p^\theta(X^T)d\Pi_\lambda(\theta)
	 		\le \max_{h\in \set{1,\ldots,N}} \sup_{\lambda\in I_h}\int \bar p^\theta(X^T)d\Pi_\lambda(\theta)
	 	\end{align*}
	 	where \(N\in\NN, I_h\subseteq \Lambda, h\in\set{1,\ldots,N}\) so that \(
	 		\bigcup_{k=1}^N I_h\supseteq \Lambda_0^c,
	 	\) choose \(\lambda_1=(\alpha_1,s_1),\ldots,\lambda_N=(\alpha_N,s_N)\in\Lambda_0^c\) and sets \(I_1,\ldots, I_N\subseteq \Lambda\), so that \(\alpha_h=\inf\set{\alpha:(\alpha,s)\in I_h}\), and \\\( \text{diam}(\set{\alpha:(\alpha,s)\in I_h})\le T^{-1.5\sqrt{\log T}-1.5}/\log T\) and \( |s-s_h|\le T^{-1.5\sqrt{\log T}-\frac{10}6}\) for all \(s\) so that \((\alpha,s)\in I_h\). The first and the second condition are trivially satisfied for case 1 and case 2 in \cref{thm:posteriorcontractiontheorem}, respectively. Note that if we chose a discrete set \(\Lambda\),  we can choose singleton sets \(I_h=\set{\lambda_h}, \lambda_h\in\Lambda_0^c\).
	 	Whatever set \(\Lambda\)  we chose,  \(N\) is bounded by \(e^{5(\log T)^{3/2}} \), %
	 	for sufficiently large \(T\). It follows that
	 	\begin{align*}
	 		& P^{\theta_0}\left[\set{\sup_{\lambda\in\Lambda_0^c}\int \bar p^\theta(X^T)d\Pi_\lambda(\theta)>e^{-BT\eps_0^2}}\cap  E\right]\\
	 		\le & e^{5(\log T)^{3/2}}\max_{1\le k\le N}P^{\theta_0}\left[\set{\sup_{\lambda\in I_h}\int \bar p^\theta(X^T)d\Pi_\lambda(\theta)>e^{-BT\eps_0^2}}\cap  E \right].
	 	\end{align*}

	 	It is only left to show that\[\max_{1\le k\le N}P^{\theta_0}\left[\set{\sup_{\lambda\in I_h}\int \bar p^\theta(X^T)d\Pi_\lambda(\theta)>e^{-BT\eps_0^2}}\cap E \right]=\littleo\rh{e^{-5(\log T)^{3/2}}}\quad\text{as}\quad T\to\infty.\]

	 	Let \(h\in\set{1,\ldots,N}\). Then \(\lambda_h\in\Lambda_0^c\) and so \(\eps_{\lambda_h}\ge M\eps_0\). Let \(\varphi\) be the test function of \cref{lem:checkingtheconditions}, with \(\eps=\eps_{\lambda_h}\) and \(U=K\) and \(\Theta\) the corresponding sieve, where \(K\) is the constant in \cref{eq:definingpropertyepslambda}, which can be chosen independent of \(h\in\set{1,\ldots,N}\).

	 	For each  \(h\in\set{1,\ldots,N}\) and for all \(\lambda\in I_h\), we define a 
	 	measurable map (which we call a transformation) 
	 	\(\Phi_{\lambda}:\sp\set{\phi_k:1\le k\le \floor T} (=\supp \Pi_{\lambda_h}) 
	 	\to \sp\set{\phi_k:1\le k\le \floor T} (=\supp \Pi_\lambda )\) such that when 
	 	\(\theta\sim \Pi_{\lambda_h}\), then \(\Phi_{\lambda}(\theta)\sim 
	 	\Pi_{\lambda}\). Obviously, when we chose a discrete set \(\Lambda \) then 
	 	\(\Phi_\lambda \) is the identity map. 
	 	 For \(\theta\in \sp\set{\phi_k:1\le k\le \floor T}\), we have 
	 	 \(\theta=\sum_{k=1}^{\floor T} \theta_k\phi_k\), for some \(\theta_1,\ldots,\theta_{\floor T}\in\re\). Define for 
	 	 \(\lambda=(\alpha,s)\), \begin{align*}
	 		\Phi_{\lambda}(\theta)=\frac{s}{s_h}\sum_{k=1}^{\floor T} 
	 		k^{\alpha_h-\alpha}\theta_k\phi_k.
	 	\end{align*} 
	 	\(\Phi_\lambda \) clearly satisfies the desired properties.
	 	Using this for a non-empty family of positive random variables,
	 	 \(Y_\gamma, \gamma\in\Gamma, \sup_{\gamma\in\Gamma}\E Y_\gamma\le \E \sup_{\gamma\in\Gamma}Y_\gamma\)
	 	  and the transformation \(\Phi_{\lambda}\), we obtain 
	 		 	
	 	\begin{align*}
	 		&P^{\theta_0}\left[\set{\sup_{\lambda\in I_h}\int \bar p^\theta(X^T)d\Pi_\lambda(\theta)>e^{-BT\eps_0^2}}\cap E \right]\\
	 		\le & \E_{\theta_0}\varphi
	 		+\E_{\theta_0}\left[\I\set{\int \sup_{\lambda\in I_h}\bar p^{\Phi_{\lambda}(\theta)}(X^T)d\Pi_{\lambda_h}(\theta)>e^{-BT\eps_0^2}}(1-\varphi)\I_E\right].
	 	\end{align*} 
	 	The first term \(\E_{\theta_0}\varphi\le e^{-C_1K^2T\eps_{\lambda_h}^2}\le e^{-C_1K^2M^2T\eps_0^2}\le e^{-M^2T\eps_0^2}\), thereby using \(\eps_{\lambda_h}\ge M\eps_0\) and choosing sufficiently large \(K\), which can and will be done uniformly over \(h\in\set{1,\ldots,N}\), hereby noting that \(C_1\) does not depend on \(K\) and on the \(h\). With the help of the Markov inequality and Fubini's theorem, the second term is bounded by 
	 	\begin{align}
	 		e^{BT\eps_0^2}\int \E_{\theta_0}\left[\sup_{\lambda\in I_h}\bar p^{\Phi_{\lambda}(\theta)}(X^T)(1-\varphi(X^T))\I_E\right]d\Pi_{\lambda_h}(\theta)\nonumber\\
	 		=e^{BT\eps_0^2}\int \E_{\theta}\left[\sup_{\lambda\in I_h}(\bar p^{\Phi_{\lambda}(\theta)}(X^T)/\bar p^\theta(X^T))(1-\varphi(X^T))\I_E\right]d\Pi_{\lambda_h}(\theta) \label{eq:upperboundaftertakingmarkovandfubini}
	 	\end{align}
	 	With the help of the Cauchy-Schwartz inequality and using  \((1-\varphi(X^T))^2\le 1-\varphi (X^T)\) (because \(1-\varphi\) takes values in \([0,1]\)), we find that \cref{eq:upperboundaftertakingmarkovandfubini} is bounded by 
	 	\begin{equation}\label{eq:upperboundaftersplittingintestpartandsuppart}
\begin{split}	 	
	 	&e^{BT\eps_0^2}\sqrt{\int \E_{\theta}\left[\sup_{\lambda\in I_h}\big(\bar p^{\Phi_{\lambda}(\theta)}(X^T)/\bar p^{\theta}(X^T)\big)^2\I_E\right]d\Pi_{\lambda_h}(\theta)}
	 	\sqrt{\int \E_{\theta}\left[(1-\varphi(X^T))\I_E\right]d\Pi_{\lambda_h}(\theta)}.	 		\end{split} 	
	 		\end{equation}
	 		For the integral in the second square root of the last expression we have the following bound:%
	 		\begin{align*}
	 			&\int \E_{\theta}\left[(1-\varphi(X^T))\I_E\right]d\Pi_{\lambda_h}(\theta)
	 			\le 	\Pi_{\lambda_h}\big[\set{\theta\in\mathring L^2(\TT):\|\theta-\theta_0\| < K\eps_{\lambda_h}}\big]\\
	 			&+\int_{\theta\in\Theta(\lambda_h):\|\theta-\theta_0\|\ge K\eps_{\lambda_h}}\E_\theta[(1-\varphi(X^T))\I_E]d\Pi_{\lambda_h}(\theta)
	 			+\Pi_{\lambda_h}(\mathring L^2(\TT)\weg\Theta(\lambda_h))\\
	 			\le & e^{-T\eps_{\lambda_h}^2}+e^{-C_2KT\eps_{\lambda_h}^2}+e^{-KT\eps_{\lambda_h}^2}
	 			\le  3 e^{-M^2T\eps_0^2},
	 		\end{align*}
	 		for sufficiently large \(K\) (note that \(K\) does also not depend on \(C_2\)) and using \(\eps_{\lambda_h}\ge M\eps_0\).
	 		
	 		Let us now consider the other integral in \cref{eq:upperboundaftersplittingintestpartandsuppart}.  
	 		In the cases in which we choose a discrete set, the left square root of \cref{eq:upperboundaftersplittingintestpartandsuppart} is simply one.
	 		When we chose a continuous interval, we first bound the inner expectation uniformly over \(\theta\in \mathring L^2(\TT)\) with \(\|\theta\|_2\le D\). 
	  \begin{restatable}{theorem}{restatabletheoremupperboundforQinalphacase}\label{restatabletheorem:upperboundforQinalphacase}
	 A constant \(T_0>0\) exists, only depending on \(\delta\) in \cref{eq:defLambda}, so that for every \(D>0\) and \(\theta\in \mathring L^2(\TT)\) with \(\|\theta\|_2\le D\) and \(T\ge T_0\),  
	\begin{align}
	&\E_{\theta}\left[\sup_{\lambda\in I_h}\rh{\bar p^{\Phi_{\lambda}(\theta)}/\bar p^{\theta}(X^T)}^2\I_E \right]
	\le 12T + 8e^{24C_\rho^2D^2T^{-3\sqrt{\log T}}}\label{eq:upperboundforQoneminusthetestfunction2familyIII}
	 \end{align}
\end{restatable}
The proof of this theorem is in \cref{proofofrestatabletheorem:upperboundforQinalphacase}.
Next, we need a bound on the `tail' of the prior, which is given in 
\begin{lemma}\label{lem:tailmassofthepriorfamiliesIandII} For \(x/s\ge \frac1{\sqrt{2\alpha}}\),  we have \begin{align*}
		\Pi_\lambda (\|\theta\|_2>x)\le &	\expa{-\alpha x^2/(s^2)}.
		\end{align*}
\end{lemma}
\begin{proof}
	Let \(\Phi\) be the cumulative distribution function of a standard normally distributed random variable.  
It follows from \cite[corollary 5.1]{vaart2008} and elementary bounds on \(\Phi \) that when \(x\ge \frac 1 {\sqrt{2\alpha}}\),
\begin{align*}
\P\rh{\norm{\sum_{k=1}^{\floor T} k^{-\alpha-1/2}Z_k\phi_k}_2>x}
= \P\rh{\sum_{k=1}^{\floor T} k^{-2\alpha-1}Z_k^2>x^2}\\
\le 	\P\rh{\sum_{k=1}^\infty k^{-2\alpha-1}Z_k^2>x^2}
	\le  1-\Phi(\sqrt{2\alpha} x)\le  \expa{-\alpha x^2}. 
\end{align*} 
This implies the result after substituting \(x/s\) for \(x\).
\end{proof}

Note that for \( \Lambda_1 \), \(x/s\ge\frac1{\sqrt{2\alpha}}\) when \(x \ge \frac1{\sqrt{1+2\delta}}T^{\sqrt{\log T}}\). %
As \( \Pi_\lambda(\|\theta\|_2>0)=e^0\), we may write
	\begin{align*}
&\int \E_{\theta}\left[\sup_{\lambda\in I_h}\big(\bar p^{\Phi_{\lambda}(\theta)}(X^T)/\bar p^{\theta}(X^T)\big)^2\I_ E\right]d\Pi_{\lambda_h}(\theta)\\
\le &\sum_{m=0}^\infty\sup_{\|\theta\|\le \sqrt{\frac{m+1}{1+2\delta}}  T^{\sqrt{\log T}} } P^{\theta}\left[\sup_{\lambda\in I_h}\big(\bar p^{\Phi_{\lambda}(\theta)}(X^T)/\bar p^{\theta}(X^T)\big)^2\I_E\right]\Pi_{\lambda_h}\rh{\|\theta\|_2\ge \sqrt{\frac{m}{1+2\delta}} T^{\sqrt{\log T}}}\\
\le &\sum_{m=0}^\infty \rh{12T+8\expa{24C_\rho^2\frac{m+1}{1+2\delta}T^{-\sqrt{\log T}}}} e^{- m/2}
=   \frac{12T}{1-e^{-1/2}} + 8 \expa{\frac{24C_\rho^2}{1+2\delta } T^{-\sqrt{\log T}}} \\
&+ 8\sum_{m=1}^\infty \expa{- m T^{-\sqrt{\log T}}\rh{ \frac12T^{\sqrt{\log T}} - \frac{24C_\rho^2}{1+2\delta}\frac{m+1}{m} }}\\ 
\le&  31T +  22 + 29    
\le  32T,
\end{align*}
for sufficiently large \(T\). 
In summary, 
\begin{align}
&\P^{\theta_0}\rh{\set{\sup_{\lambda\in \Lambda_0^c}\int \bar p^\theta(X^T)d\Pi_\lambda(\theta) > e^{-BT\eps_0^2}}\cap  E}
\le   e^{5(\log T)^{3/2}}\rh{e^{-M^2T\eps_0^2}+\sqrt{96T}e^{(B-M^2/2)T\eps_0^2}}.\label{eq:upperboundconclusion}  
\end{align}
Using \(T\eps_0^2\gtrsim e^{\frac12\sqrt{\log T}}\) (\cref{lem:Tepszerosquaredconvergestoinfinity}),  which is eventually larger than \(5(\log T)^{3/2}\), it follows that for \(M>\sqrt{2B}\), \cref{eq:upperboundconclusion} converges to zero, which completes the proof of the theorem.
\end{proof}	
	\section{Posterior convergence on the favourable event}\label{sec:posteriorconvergencegivenhatlambdainLambdazero}
	\begin{theorem}
	Let \(F\) be the event and \(M\) be the constant of \cref{thm:hatlambdainafavourableset}. 		For some constant \(H>1\), 
			 \begin{align*}
	 	\E_{\theta_0}\left[\Pi_{\hat\lambda}(\|\theta-\theta_0\|_2\ge HM\eps_{0}\mid X^T)\I_F\right]
	 \end{align*}
	 converges to zero, as \(T\to\infty\). 
	\end{theorem}
	\begin{proof}
	Choose \(\lambda'\in\Lambda_0\) satisfying \(\eps_{\lambda'}\le 2\eps_0\). 
	It follows from \(\int \bar p^\theta(X^T)d\Pi_{\hat\lambda}(\theta)\ge\frac12  \int \bar p^\theta(X^T)d\Pi_{\lambda'}(\theta)\) and \cref{thm:hatlambdainafavourableset} that in the event \(F\) (with \(A\) defined in \cref{thm:hatlambdainafavourableset}),
	\begin{align*}
	&\Pi_{\hat\lambda}(\|\theta-\theta_0\|_2\ge HM\eps_{0}\mid X^T)
\le 	2	e^{AT\eps_0^2}\int_{\|\theta-\theta_0\|_2\ge HM\eps_{0}}\bar p^{\theta}(X^T)d\Pi_{\hat\lambda}(\theta).
	\end{align*}
	There are subsets \(I_h', k=1,\ldots,N'\) of \(\Lambda\), containing a point \(\lambda_h'\in\Lambda_0\), which is defined as in the proof of  \cref{thm:hatlambdainafavourableset} with \(\Lambda_0^c\) replaced by \(\Lambda_0\) and where \(N'\) satisfies the same upper bound as \(N\), such that 
		\(
	\Lambda_0\subseteq \bigcup_{k=1}^{N'}I_h'.
	\)
	Let \(\Phi_\lambda'\) be defined  as \(\Phi_\lambda\) in the proof of  \cref{thm:hatlambdainafavourableset}, with \(\lambda_h\) replaced by \(\lambda_h'\) and \(I_h\) replaced by \(I_h'\). 
	Using \(M\eps_0\ge\eps_{\lambda_h'}\), for every \(h\in\set{1,\ldots,N'}\), we can bound for a constant \(A'>A\), 
	\begin{align*}
		&P^{\theta_0}\left[\int_{\|\theta-\theta_0\|_2\ge HM\eps_0}\bar p^{\theta}(X^T)d\Pi_{\hat\lambda}(\theta)>e^{-A'T\eps_0^2},F\right]\\
		\le & N'\cdot\max_{h\in\set{1,\ldots,N'}}P^{\theta_0}\left[\int_{\|\theta-\theta_0\|\ge H\eps_{\lambda_h'}}\sup_{\lambda\in I_h}\bar p^{\Phi_{\lambda}'(\theta)}(X^T)d\Pi_{\lambda_h}(\theta)>e^{-A'T\eps_0^2},F\right].
	\end{align*}
	Let \(h\in\set{1,\ldots,N'}\). Let \(\varphi\) be the test functions of \cref{lem:checkingtheconditions} with \(\lambda=\lambda_h', \eps=\eps_{\lambda_h}\) and \(U=H\ge K\). Then for  sufficiently large \(T\), 		
	 \begin{align*}
	 	&P^{\theta_0}\left[\int_{\|\theta-\theta_0\|\ge H\eps_{s_h'}}\sup_{\lambda\in I_h'}\bar p^{\Phi_{\lambda}'(\theta)}(X^T)d\Pi_{\lambda_h'}(\theta)>e^{-A'T\eps_0^2},F\right]\\
	 	\le & \E_{\theta_0}\varphi(X^T)
	 	+e^{A'T\eps_0^2}\int_{\|\theta-\theta_0\|\ge H\eps_{\lambda_h'}}\E_{\theta_0}\left[\sup_{\lambda\in I_h'}\bar p^{\Phi_{\lambda}'(\theta)}(X^T)(1-\varphi(X^T))\I_F\right]d\Pi_{\lambda_h}(\theta)\\
	 	\le & e^{-H^2T\eps_0^2}+ e^{A'T\eps_0^2}\sqrt{19T\rh{e^{-C_2H^2\eps_0^2}+  e^{-H^2T\eps_0^2}}},
	 \end{align*}
in a similar way as in the proof of \cref{thm:hatlambdainafavourableset} and using that \(\eps_\lambda\ge \eps_0\).  Hence, for \(H>\sqrt{\frac{2A'}{C_2\wedge1}}\), \(\Pi_{\hat\lambda}(\|\theta-\theta_0\|_2\ge HM\eps_{0}\mid X^T)\to 0\), as \(T\to\infty\), in a similar way as in the proof of \cref{thm:hatlambdainafavourableset}.%
\end{proof}	 
\section{The proof of \cref{restatabletheorem:upperboundforQinalphacase}}\label{proofofrestatabletheorem:upperboundforQinalphacase}
\restatabletheoremupperboundforQinalphacase*
\begin{proof}
Note that \(W\) is a BM under \(P^{\theta_0}\), and so 
\begin{align*}
W^\theta_t:=W_t-\int_0^t(\theta(X_u)-\theta_0(X_u))du
\end{align*}
is a BM under \(P^{\theta}\).
We may write 
\begin{align*}
	&\bar p^{\Phi_{\lambda}(\theta)}
= \bar p^\theta (X^T)
 \expa{\int_0^T(\Phi_{\lambda}(\theta)(X_u)-\theta(X_u))dW_u^\theta-\frac12\int_0^T\big(\Phi_{\lambda}(\theta)(X_u)- \theta(X_u)\big)^2du}.
\end{align*}
Note that for all \(G>0 \)
\begin{align*}
&\E_{\theta}\left[\sup_{\lambda\in I_h}\rh{\bar p^{\Phi_{\lambda}(\theta)}/\bar p^{\theta}(X^T)}^2\I_E \right]
\le e^{2G}+\sum_{g=G}^\infty e^{2g+2}\P^{\theta}\left(\ah{\frac{\sup_{\lambda\in I_h}\bar p^{\Phi_{\lambda}(\theta)}}{\bar p^{\theta}(X^T)}>e^g}\cap  E\right).
\end{align*}
Let \(D>0\) and \(\theta\in\mathring  L^2(\TT)\) with \(\|\theta\|_2\le D\). Note that in the event \(E\), 
\begin{align*}
&\vh{ \int_0^{\sdot}(\Phi_{\lambda}(\theta)(X_t)-\theta(X_t))dW_t^\theta}_T
=  \int_0^{T}(\Phi_{\lambda}(\theta)(X_t)-\theta(X_t))^2dt
\le  C_\rho^2 T \|\Phi_{\lambda}(\theta)-\theta\|_2^2.
\end{align*}
We have \begin{align*}
&\|\Phi_{(\alpha,s)}(\theta)-\theta\|_2^2
=\sum_{k=1}^{\floor T} \rh{\frac{s}{s_h}k^{\alpha_h-\alpha}-1}^2\inpr{\theta,\phi_k}^2.
\end{align*}
We next bound \(\abs{\frac {s}{s_h}k^{\alpha_h-\alpha}-1}\).  For this, we use the following
\begin{lemma}\label{lem:boundfor1minusmuetothepowerminusx}
	For \(\mu>0\) and \(x\ge0\) we have \(
	1-\mu e^{-x}\le \mu x + 1-\mu.
	\)
	Moreover, for all \(\mu>0\) and \(x\ge 0\), we have \(
	|1-\mu e^{-x}|\le \mu x + |1-\mu|.
	\)
\end{lemma}
\begin{proof}
	Let us consider the first inequality. 
	Denote \(f(x)=1-\mu e^{-x}\) and \(g(x)=\mu x + 1-\mu\). Note that \(f(0)=g(0)\) and \(g'(x)=\mu\ge f'(x)=\mu e^{-x}\), for all \(x\ge 0\). It follows that \(g(x)\ge f(x)\) for all \(x\ge0\). This proves the first inequality. 
	For the second inequality, we note that when \(\mu>e^x\), \(|1-\mu e^{-x}|=\mu e^{-x}-1\le \mu-1\le \mu x + |1-\mu|\). When \(\mu\le e^x\), we have \(|1-\mu e^{-x}|=1-\mu e^{-x}\le \mu x + 1-\mu \le  \mu x + |1-\mu|.\) The proof is now complete.
\end{proof}
Using lemma \cref{lem:boundfor1minusmuetothepowerminusx}, and the fact that \(\alpha_h\le \alpha\), we see that \begin{equation}\label{eq:pointwhereIneedtotruncateatT}
	|1-\frac{s}{s_h}k^{\alpha_h-\alpha}|
\le  \frac{s}{s_h}|\alpha-\alpha_h|\log k+\frac{|s-s_h|}{s_h}
\le   \frac{s}{s_h}|\alpha-\alpha_h|\log T+\frac{|s-s_h|}{s_h}.
\end{equation} 
Using the lower bound for \(s\in\Lambda\) and the upper bound for \(|s-s_h|\) on \(I_h\), we derive\begin{align*}
\frac{|s-s_h|}{s_h}
\le  T^{\frac1{4+4\alpha}}T^{-1.5\sqrt{\log T}-\frac{10}{6}}\le T^{\frac1{6+4\delta}-\frac16} T^{-1.5\sqrt{\log T}-1.5}\le 0.5T^{-1.5\sqrt{\log T}-1.5}, 
\end{align*}
for sufficiently large \(T\). 
We also have \begin{align*}
\frac{s}{s_h}\le &1+\frac{|s-s_h|}{s_h}\le 1.5, 
\end{align*}
for sufficiently large \(T\). 
It follows that for all \(\lambda=(\alpha,s)\) and all \(k\) that we consider, using the upper bound on \(|\alpha_h-\alpha|\) on \(I_h\), \[
|1-\frac{s}{s_h}k^{\alpha_h-\alpha}|\le  2T^{-1.5\sqrt{\log T}-1.5},
\]
for sufficiently large \(T\).

By the Cauchy inequality and the assumption on \(\theta\), we have \(|\inpr{\theta, \phi_k}|\le \|\theta\|_2\le D\). 

Using the Bernstein inequality (\cite[pp. 153-154]{revuzyor1999}), for \(g>0\), we obtain 
\begin{align*}
&\P^{\theta}\left[\frac{\sup_{\lambda \in I_h}\bar p^{\Phi_{\lambda}(\theta)}}{\bar p^{\theta}(X^T)}>e^g, E\right]
\le  \P\vh{\sup_{\lambda\in I_h}\int_0^T \big(\Phi_{\lambda}(\theta)(X_t)- \theta(X_t)\big)dW_t^\theta > g, E}\\
\le & \P\vh{T\sup_{(\alpha,s)\in I_h}
\max_{k\in\set{1,\ldots,\floor{T}}} \abs{\rh{\frac{s}{s_h}k^{\alpha_h-\alpha}-1}\inpr{\theta,\phi_k}\int_0^T\phi_k(X_t)dW_t^\theta} > g
, E
}\\
\le & T \max_{k\in\set{1,\ldots,\floor{T}}} \P\vh{\abs{\int_0^T \phi_k(X_t)dW_t^\theta } > \frac{g}{2D}T^{1.5\sqrt{\log T}+0.5} , \vh{\int_0^{\sdot } \phi_k(X_t)dW_t^\theta }_T\le C_\rho^2 T  }\\
\le & T\expa{-\frac{g^2}{4C_\rho^2D^2}T^{3\sqrt{\log T}} }.
\end{align*}

Note that for \(g>0\), \(
Ag^2 - 2g \ge g\quad 
\desda\quad
g \ge 3A^{-1}\).
It follows that for integers \(G\ge 12C_\rho^2D^2T^{-3\sqrt{\log T}}\), \begin{align*}
	&\E_{\theta}\left[\sup_{s\in I_h}\rh{\bar p^{s\theta/s_h}(X^T)/\bar p^{\theta}(X^T)}^2, E \right]
\le e^{2G}+e^2T\sum_{g=G}^\infty e^{-g}
= e^{2G}+e^2T\frac{e^{-G}}{1-e^{-1}}
\le  e^{2G}+ 12T,
\end{align*}
as \(e^2\frac{e^{-G}}{1-e^{-1}}\le \frac{e^2}{1-e^{-1}}\le 12\). 
Taking  \(G=\ceil{ 12C_\rho^2D^2T^{-3\sqrt{\log T}}}\), and using \(\ceil x\le x+1\),  for every real number \(x\) and \(e^2\le 8\), the last display is further bounded by \(
12T + 8e^{24C_\rho^2D^2T^{-3\sqrt{\log T}}}.
\)
\end{proof}
\section{Upper and lower bounds for \(\eps_\lambda\)}\label{sec:upperandlowerboundsforepslampda} 
Let \(\Pi_\lambda^\infty\) be the distribution on \(\mathring L^2(\TT)\) defined by the law of\[
s\sum_{ k=1}^\infty k^{-\alpha-1/2}Z_k\phi_k. 
\]

	It follows from \cite[lemmas 4.2, 4.3 and 4.4]{waaijzanten2016} and \cite[lemma 5.3]{vaart2008} that for sufficiently small \(\eps/s>0\), \(\beta\le\alpha+1/2\) and 
	\(\theta_0\in\dot H^\beta(\TT)\), 
	\begin{align*}
	{\expa{-C_1\rh{C_2 \alpha (s/\eps)^{1/\alpha}+\frac1{s^2}\eps^{\frac{2\beta-2\alpha-1}{\beta}}}}  \le  	\Pi_\lambda^\infty(\|\theta-\theta_0\|_2<\eps) \le  \expa{-C_3 \alpha  (s/\eps)^{1/\alpha}}}.
	\end{align*} 
	As in the proof of \cref{lem:tailmassofthepriorfamiliesIandII}, we have \begin{equation}
	\begin{split}
	\P\rh{\norm{\sum_{k=1}^{\floor T} k^{-1/2-\alpha}Z_k\phi_k}_2<\eps }\ge \P\rh{\norm{\sum_{k=1}^{\infty } k^{-1/2-\alpha}Z_k\phi_k}_2<\eps }. 
	\end{split}
	\end{equation}
	For an inequality in the reverse order, we need the following modification of Markov's inequality
	\begin{lemma}\label{lem:modmarkov}
	 Let \(r>0\), \(X\) be a non-negative random variable, and \(A\) be an event, such that \(A\) and \(X\) are independent. Then \(
	 \P(A\cap \set{X\ge r}) \le \frac1r \P(A)\E X. 
	 \)
	\end{lemma}
\begin{proof}
	The lemma follows from \(
	\P(A\cap\set{X\ge r})= \E \vh{\I_A \I_{\set{X\ge r}} }\le \E\vh{\I_A \frac Xr}= \frac1r \P(A)\E X, 
	\) by independence.
\end{proof}
		Note that 
	\begin{equation}
	\E\vh{\norm{\sum_{k=\floor T + 1}^{\infty} k^{-1/2-\alpha}Z_k\phi_k}_2^2 }=  \sum_{k=\floor T + 1}^{\infty} k^{-1-2\alpha}
	\le \frac1{2\alpha} T^{-2\alpha}. 
	\end{equation}
	Hence for \(\eps>\frac1{\sqrt{\alpha}}T^{-\alpha}\), an application of \cref{lem:modmarkov} gives, \begin{align*}
	&\P\rh{\norm{\sum_{k=1}^{\floor T} k^{-1/2-\alpha}Z_k\phi_k}_2<\eps }
	=  \P\rh{\norm{\sum_{k=1}^{\floor T} k^{-1/2-\alpha}Z_k\phi_k}_2<\eps, \norm{\sum_{k=\floor T +1}^\infty k^{-1/2-\alpha}Z_k\phi_k}_2<\eps }\\
	&+\P\rh{\norm{\sum_{k=1}^{\floor T} k^{-1/2-\alpha}Z_k\phi_k}_2<\eps, \norm{\sum_{k=\floor T +1}^\infty k^{-1/2-\alpha}Z_k\phi_k}_2\ge\eps }\\
	\le & \P\rh{\norm{\sum_{k=1}^\infty k^{-1/2-\alpha}Z_k\phi_k}_2<2\eps }
	+\eps^{-2}\frac1{2\alpha}T^{-2\alpha}\P\rh{\norm{\sum_{k=1}^{\floor T} k^{-1/2-\alpha}Z_k\phi_k}_2<\eps }\\
	\le &\P\rh{\norm{\sum_{k=1}^\infty k^{-1/2-\alpha}Z_k\phi_k}_2<2\eps }
	+\frac12\P\rh{\norm{\sum_{k=1}^{\floor T} k^{-1/2-\alpha}Z_k\phi_k}_2<\eps }
	\intertext{Rearranging the terms gives}
	&\P\rh{\norm{\sum_{k=1}^{\floor T} k^{-1/2-\alpha}Z_k\phi_k}_2<\eps }
	\le   2\P\rh{\norm{\sum_{k=1}^\infty k^{-1/2-\alpha}Z_k\phi_k}_2<2\eps }.
	\end{align*}	
	Hence, for sufficiently small \(\eps/s>\frac1{\sqrt\alpha}T^{-\alpha} \), \(\beta\le\alpha+1/2\), and \(\theta_0\in\dot H^\beta(\TT)\), 
	\begin{equation}\label{eq:onderenbovengrenzenvankleinebalkansen}
		{\expa{-C_1\rh{C_2\alpha (s/\eps)^{1/\alpha}+\frac1{s^2}\eps^{\frac{2\beta-2\alpha-1}{\beta}}}} \le  \Pi_\lambda(\|\theta-\theta_0\|_2<\eps)  \le 2\expa{-C_3\alpha (s/(2\eps))^{1/\alpha}}}.
	\end{equation} 
	We have 
	\begin{align*}
		2\expa{-C_3\alpha (s/(2K\eps_\lambda))^{1/\alpha}}\ge \Pi_\lambda (\|\theta-\theta_0\|_2<K\eps_\lambda)=e^{-T\eps_\lambda^2}.
	\end{align*}
	It follows that for sufficiently small \(\eps_\lambda/s\) 
	\begin{equation}\label{eq:lowerboundonepslambda}
	\eps_\lambda\ge  2^{-\frac{1+\alpha}{1+2\alpha}}C_3^{\frac{\alpha}{1+2\alpha}}K^{-\frac1{1+2\alpha}}\alpha^{\frac\alpha{1+2\alpha}}T^{-\frac\alpha{1+2\alpha}}s^{\frac1{1+2\alpha}}.
	\end{equation}
	Using this lower bound we derive the lower bound \begin{equation}
	\inf_{\lambda\in \Lambda_1} T\eps_\lambda^2
	\ge  \tilde C_3 K^{-1}e^{\frac12\sqrt{\log T}},\label{lem:Tepszerosquaredconvergestoinfinity} 
	\end{equation}
	where \(\tilde C_3 \) is a constant.
Using the other bound of \cref{eq:onderenbovengrenzenvankleinebalkansen},  we obtain	\begin{align*}
		&\expa{-C_1\rh{C_2\alpha (s/(K\eps_\lambda))^{1/\alpha}+\frac1{s^2}(K\eps_\lambda)^{\frac{2\beta-2\alpha-1}{\beta}}}}
		\le  \Pi(\|\theta-\theta_0\|_2<K\eps_\lambda)
		=  e^{-T\eps_\lambda^2}.
	\end{align*}
	
	Hence, either \(C_1C_2\alpha (s/(K\eps_\lambda))^{1/\alpha}\ge T\eps_\lambda^2/2\) or \(C_1\frac1{s^2}(K\eps_\lambda)^{\frac{2\beta-2\alpha-1}{\beta}}\ge T\eps_\lambda^2/2\). It follows that  
\begin{align*}
 \eps_\lambda \le 2^{\frac\alpha{1+2\alpha}}C_1^{\frac\alpha{1+2\alpha}}C_2^{\frac\alpha{1+2\alpha}}K^{-\frac1{1+2\alpha}} \alpha ^{\frac\alpha{1+2\alpha}} s^{\frac1{1+2\alpha}}T^{-\frac\alpha{1+2\alpha}}\bigvee 2^{\frac{\beta}{1+2\alpha}}C_1^{\frac{\beta}{1+2\alpha}}K^{\frac{2\beta-2\alpha-1}{1+2\alpha}}s^{-\frac{2\beta}{1+2\alpha}}T^{-\frac\beta{1+2\alpha}}.
\end{align*}	
		For each of the  three cases of \cref{thm:posteriorcontractiontheorem}, we derive an upper bound for \(\eps_0\) separately.
	
	\begin{itemize}
		\item In case 1, \( \alpha \) is fixed, so \begin{align*}
		\eps_\lambda \lesssim s^{\frac1{1+2\alpha}}T^{-\frac{\alpha}{1+2\alpha}} \bigvee s^{-\frac{2\beta}{1+2\alpha}} T^{-\frac\beta{1+2\alpha}}.
		\end{align*}
		We see that the best possible upper bound for \(\eps_0\) (up to a constant) is attained when  
		\[
		s^{\frac1{1+2\alpha}} T^{-\frac\alpha{1+2\alpha}} \asymp s^{-\frac{2\beta}{1+2\alpha}}T^{-\frac\beta{1+2\alpha}},\]
		that is, when \(s\asymp T^{\frac{\alpha-\beta}{1+2\beta}}\).  Note that  for \( \beta\in (0,\alpha+1/2]\) a quantity proportional to \( T^{\frac{\alpha-\beta}{1+2\beta}}\) is in \(\Lambda\) (also in the discrete case), and hence,  \(\eps_0\lesssim T^{-\frac\beta{1+2\beta}}\). %
		\item In case 2, \(s\) is fixed, so \[
		\eps_\lambda \lesssim  \alpha^{\frac\alpha{1+2\alpha}} T^{-\frac{\alpha}{1+2\alpha}}\bigvee T^{-\frac{\beta}{1+2\alpha}}.
		\]
		We see that the best possible upper bound for \(\eps_0\) (up to a constant) is attained when \(\alpha=\beta+\mathcal O(1/\log T)\). For sufficiently large \(T\), \(\Lambda\) contains such an element, and in the discrete case, when \(\beta\ge\alpha+\delta\), in which case \(\eps_0\lesssim T^{-\frac\beta{1+2\beta}}\). 
		\item In case 3, we have \[
		\eps_\lambda\lesssim \alpha^{\frac\alpha{1+2\alpha}} s^{-\frac{1}{1+2\alpha}}T^{-\frac{\alpha}{1+2\alpha}} \bigvee s^{-\frac{2\beta}{1+2\alpha}}T^{-\frac\beta{1+2\alpha}}. 
		\] We see that the best possible upper bound for \(\eps_0\) (up to a constant) is attained when  \(\alpha\ge (\beta-1/2)\vee (1/2+\delta)\) and 
		\[
		T^{-\frac\alpha{1+2\alpha}}s^{\frac1{1+2\alpha}}\asymp s^{-\frac{2\beta}{1+2\alpha}}T^{-\frac\beta{1+2\alpha}},\]
		that is, when \(s\asymp T^{\frac{\alpha-\beta}{1+2\beta}}\). Such a pair \((\alpha,s)\) is in \(\Lambda\) (for sufficiently large \(T\)), and in the discrete case, and we have \(\eps_0\lesssim T^{-\frac\beta{1+2\beta}}\). 
	\end{itemize}
	
\section{The existence of test functions}\label{sec:tests}

\begin{lemma}\label{lem:checkingtheconditions}
	There are positive constants \(C_1,C_2\) and \(K\), only depending on \(c_\rho\) and \(C_\rho\), such that for every \(\lambda\in\Lambda\) and \(U\ge K\), there are measurable sets (sieves)  \(\Theta_\lambda\subseteq \mathring L^2(\TT)\) satisfying
	\begin{align}
		\Pi_\lambda(\mathring L^2(\TT)\weg\Theta_\lambda)\le e^{-U^2T\eps_\lambda^2},\label{eq:remainingmassinexistenceoftestlemma}
	\end{align}
	and  measurable maps \(\varphi_\lambda:C[0,T]\to\set{0,1}\), which satisfy
	\begin{align}
		\E_{\theta_0}\varphi_\lambda(X^T) \le & e^{-C_1U^2T\eps_\lambda^2},\label{eq:proofofexistenceoftestserroroffirstorder}
		\intertext{and for all \(\theta\in\Theta_\lambda,\|\theta-\theta_0\|_2\ge U\eps_\lambda\),}
		\E_\theta[(1-\varphi_\lambda(X^T))\I_E]\le & e^{-C_2U^2T\eps_\lambda^2}.\label{eq:proofofexistenceoftestserrorofsecondorder}
	\end{align}
\end{lemma}
\begin{proof} In this proof, let \(N(\eps,S,d)\) denote the minimal number of balls of size \(\eps\) in a \(d\)-metric needed to cover a set \(S\). The log of this number is referred to as the entropy. It follows from \cite[lemma 4.1]{meulen2006} and the equivalence of the \(h\)- and the \(L^2\)-metric on \(E\) (\cref{eq:definitionofE}) that there are constants \(C_1\), \(C_2\), and \(c>0\), only depending on \(c_\rho\) and \(C_\rho\), such that when \(\Theta_\lambda\) is a measurable set satisfying the entropy condition \begin{equation}\label{eq:empiricalbayessufficientconditionforexistingoftestsintermsofentropy}
		\log N\rh{\frac{c_\rho\eps_\lambda}{8C_\rho},\set{\theta\in \Theta_\lambda:\|\theta-\theta_0\|_2\le\eps_\lambda},\|\sdot\|_2}\le c^2U^2T\eps_\lambda^2,
	\end{equation}
	for \(U\ge K\), then there are measurable maps \(\varphi_\lambda\) (depending on \(\lambda,T\) and \(U\)) taking values in \(\{0,1\}\) which satisfy \cref{eq:proofofexistenceoftestserroroffirstorder,eq:proofofexistenceoftestserrorofsecondorder}.

We continue by showing that  such \(\Theta_\lambda\) actually exist and also satisfy \cref{eq:remainingmassinexistenceoftestlemma}. For this, we follow \cite[\S 6.2]{waaijzanten2016}. Every \(\theta\in\mathring L^2(\TT)\) has an expansion \(\theta=\sum_{k=1}^\infty \theta_k\phi_k\) in the chosen orthonormal basis of \(\mathring L^2(\TT)\). Recall that in the aforementioned study,  \(\mathring H_1^{\alpha+1/2}\) is the set of \(\theta\in \mathring L^2(\TT)\) for which \(\sum_{k=1}^\infty k^{2\alpha+1}\theta_k^2\le 1\) and \(\mathring L^2_1(\TT)\) is the closed unit ball in \(\mathring L^2(\TT)\). We take the following sieves,	\begin{equation*}
		\Theta_\lambda=R\mathring H_1^{\alpha+1/2}(\TT)+\frac{c_\rho\eps_\lambda}{16C_\rho} \mathring L^2_1(\TT).
	\end{equation*}
Then, for some constant \(C>0\) (only depending on \(\lambda,c_\rho\) and \(C_\rho\)), 
	\begin{align*}
\Pi_s(\mathring L^2(\TT)\weg\Theta_\lambda)\le \expa{-\frac12\rh{\frac Rs-\sqrt{C(s/\eps_\lambda)^{1/\alpha}}}^2}.
	\end{align*}
	Let 
	\begin{align*}
		R &=s\rh{U\sqrt{2T}\eps_\lambda+\sqrt{C(s/\eps_\lambda)^{1/\alpha}}},
	\end{align*}
	then \(\Pi_s(\mathring L^2(\TT)\weg\Theta_\lambda)\le e^{-U^2T\eps_\lambda^2}\), hence \cref{eq:remainingmassinexistenceoftestlemma} is satisfied. 
	For some constant \(\hat C>0\) (depending on \(\lambda,c_\rho\), and \(C_\rho\)), the entropy is bounded as follows 
	\begin{equation*}
	\log N\rh{\frac{c_\rho\eps_ s}{8C_\rho},\set{\theta\in \Theta_ s:\|\theta-\theta_0\|_2\le\eps_ s},\|\sdot\|_2}\le \hat C\rh{\frac R{\eps_\lambda}}^{2/(1+2\alpha)}.\end{equation*}
	For some constant \(\tilde C>0\), this is bounded by
	\begin{align*}
		&\tilde C\rh{s^{\frac2{1+2\alpha}}U^{\frac2{1+2\alpha}}T^{\frac1{1+2\alpha}}\bigvee s^{1/\alpha}\eps_\lambda^{-1/\alpha}}\\
	= & \tilde C T\eps_\lambda^2\rh{s^{\frac2{1+2\alpha}}U^{\frac2{1+2\alpha}}T^{-\frac{2\alpha}{1+2\alpha}}\eps_\lambda^{-2}\bigvee T^{-1}s^{1/\alpha}\eps_\lambda^{-\frac{1+2\alpha}\alpha}}
	\end{align*}
	 Inserting the lower bound of \cref{eq:lowerboundonepslambda} shows that this is bounded for some constant \(\bar C>0\), (only depending on \(c_\rho,C_\rho\), and \(\lambda\)) by 	 \begin{align*}
	 	\bar C T\eps_\lambda^2\rh{(UK)^{\frac2{1+2\alpha}}\bigvee K^{1/\alpha}}.
	 \end{align*}
	 This is for \(U\ge K\) bounded by \begin{equation}\label{eq:pointwherdeltaisneeded}
	 \bar  CK^{\frac{2-4\alpha}{1+2\alpha}}U^2T\eps_\lambda^2\le \bar  CK^{\frac{-4\delta}{2+2\delta}}U^2T\eps_\lambda^2\le  cU^2T\eps_\lambda^2,
	 \end{equation}
	  for sufficiently large \(K\), which establish \cref{eq:empiricalbayessufficientconditionforexistingoftestsintermsofentropy}. 
\end{proof}
	\section{Discussion}
	We examined posterior contraction rates for priors whose hyper-parameter(s) are estimated using the marginal maximum likelihood estimator. It is shown that estimating the smoothness and scaling parameter at the same time is advantageous in terms of the guaranteed adaptive range.
	
	 It would be nice to investigate ways to work around the requirement that \(\alpha\) is bounded away from a half, which seems artificial to us, because in hierarchical Bayes one obtains adaptivity to every \(\beta>0\) by putting a specific prior on \(\alpha\), which is supported on \([0, \log T]\), see \cite[theorem 3.4]{waaijzanten2016}. We do need the condition  \(\alpha\ge 1/2+\delta\) to satisfy the entropy condition in the proof of the existence of tests, \cref{eq:pointwherdeltaisneeded}.  A similar phenomenon occurs in \cite[proposition 3.5]{rousseauszabo2017}, where they need \(\alpha>1/2 \) and \(\beta>1/2\) for the same prior in the log-linear model. This may be due to  the fact that in the chosen approach for  the proofs, one not only needs a lower bound on the amount of prior mass around the true parameter, as in hierarchical Bayes, but also an upper bound on the prior mass, see \cref{eq:definingpropertyepslambda} in this study and equations 2.1 and 2.2 in \cite{rousseauszabo2017}.
	
	It may also be worthwhile investigating rates of convergence for the MLL estimator on the number of basis functions \(J\), instead of using a fixed \(J=\floor T \) basis functions, probably in combination with an estimator of the scaling parameter. In the hierarchical approach truncating the prior at a random \(J\), via a hyperprior on \(J\in\NN\) (in combination with a prior on the scaling), gave both good asymptotic and numerical  results in \cite{waaijzanten2017} and \cite{Moritz}, respectively. 
	
	In \cite{waaijzanten2016}, the authors use  the infinite series \cref{eq:seriesexpansion}, while we truncate at \(\floor T\). The reason is to control the uniform convergence over all posteriors in \(\Lambda_0\). For this we use the bound \cref{eq:pointwhereIneedtotruncateatT}, which grows to infinity when the number of basis functions grows to infinity. The advantage of \cref{eq:seriesexpansion} is only theoretical, the prior with the law of \cref{eq:seriesexpansion} does not depend on time, which makes the posterior a martingale. For case 1 with a continuous set \(\Lambda\), and in cases 1, 2, and 3 with a discrete set in \cref{thm:posteriorcontractiontheorem}, one can use an infinite series prior, as \(|\alpha-\alpha_h|=0\) in that case. 
	However, for simplicity, we decided to omit this. This is available from the author on request. In addition, due to limited computer memory one has to truncate in applications. 
\paragraph{Acknowledgment}
The author is grateful for the helpful comments from Harry van Zanten regarding an earlier version of this paper.
\end{document}